\documentclass[sn-mathphys-num]{sn-jnl}% Math and Physical Sciences Numbered Reference Style 
\usepackage{graphicx}%
\usepackage{multirow}%
\usepackage{amsmath,amssymb,amsfonts}%
\usepackage{amsthm}%
\usepackage{mathrsfs}%
\usepackage[title]{appendix}%
\usepackage{xcolor}%
\usepackage{textcomp}%
\usepackage{manyfoot}%
\usepackage{booktabs}%
\usepackage{algorithm}%
\usepackage{algorithmicx}%
\usepackage{algpseudocode}%
\usepackage{listings}%
\theoremstyle{thmstyleone}%
\newtheorem{theorem}{Theorem}%  meant for continuous numbers
\newtheorem{proposition}[theorem]{Proposition}% 

\theoremstyle{thmstyletwo}%
\newtheorem{example}{Example}%
\newtheorem{remark}{Remark}%

\theoremstyle{thmstylethree}%

\newtheorem{lemma}{Lemma}

\begin{document}

\title[Article Title]{Symmetric form  geometric constant related to isosceles orthogonality in Banach spaces}


\author[1]{\fnm{Qichuan} \sur{Ni}}\email{080821074@stu.aqnu.edu.cn}
\equalcont{These authors contributed equally to this work.}

\author*[1]{\fnm{Qi} \sur{Liu}}\email{liuq67@aqnu.edu.cn}

\author[1]{\fnm{Yuxin} \sur{Wang}}\email{y24060028@stu.aqnu.edu.cn}
\equalcont{These authors contributed equally to this work.}

\author[1]{\fnm{Jinyu} \sur{Xia}}\email{y23060036@stu.aqnu.edu.cn}
\equalcont{These authors contributed equally to this work.}

\author[1]{\fnm{Ranran} \sur{Wang}}\email{15922398780@163.com}
\equalcont{These authors contributed equally to this work.}

\affil[1]{School of Mathematics and Physics, Anqing Normal University, Anqing 246133, P. R. China}


\abstract{In this article, we introduce a novel geometric constant $L_X(t)$, which provides an equivalent definition of the von Neumann-Jordan constant from an orthogonal perspective. First, we present some fundamental properties of the constant $L_X(t)$ in Banach spaces, including its upper and lower bounds, as well as its convexity, non-increasing continuity. Next, we establish the identities of $L_X(t)$ and the function $\gamma_X(t)$, the von Neumann-Jordan constant, respectively. We also delve into the relationship between this novel  constant and several renowned geometric constants (such as the James constant and the modulus of convexity). Furthermore, by utilizing the lower bound of this new constant, we characterize Hilbert spaces. Finally, based on these findings,  we further investigate the connection between this novel  constant and the geometric properties of Banach spaces, including uniformly non-square, uniformly normal structure, uniformly smooth, etc.}




\maketitle

\section{Introduction}\label{sec1}

	The investigation into the geometric properties of Banach spaces represents a crucial domain within the broader discipline of functional analysis. In this theory, the geometric properties of normed spaces often play an important role, such as uniformly convex, uniformly non-square, etc. Quantifying the geometric constants of the geometric features of normed spaces is very useful for studying geometric properties. For instance, Clarkson introduced the notion of the modulus of convexity to describe uniformly convex spaces \cite{05}, and the von Neumann-Jordan constant  to describe Hilbert spaces and uniformly non-square spaces \cite{00}. 

In Euclidean geometry, the concept of orthogonality is indispensable. On one hand, it is manifested in the fourth axiom of Euclidean geometry, and on the other hand, it plays a crucial role in the Pythagorean theorem. However, Banach space geometry is significantly different from Euclidean geometry because there is no unique concept of orthogonality in Banach space geometry. With the development of Banach space geometry, many different orthogonalities have been introduced into general normalized linear spaces. For instance,
James \cite{06} introduced  isosceles orthogonality ($\perp_{I}$) and Pythagorean orthogonality ($\perp_{P}$): 
$$x \perp_{I} y \text{ if and only if } \|x+y\| = \|x-y\|,$$
and $$x \perp_{P} y \text{ if and only if } \|x-y\|^2 = \|x\|^2+\|y\|^2.$$

In recent years, many scholars have devoted themselves to the definition and in-depth investigation of orthogonal geometric constants. Within the realm of Banach space geometry, orthogonal geometric constants are becoming increasingly significant.
In 2022, based on the parallelogram law and isosceles orthogonality,  Liu et al. \cite{03} proposed the orthogonal geometric constant $\Omega^{\prime}(X)$ as follows: 
$$
\Omega^{\prime}(X)=\sup \left\{\frac{\|2 x+y\|^2+\|x+2 y\|^2}{\|x+y\|^2}: x, y \in X,(x, y) \neq(0,0), x\perp_{I}y\right\} .
$$

They demonstrate that in a finite dimensional Banach space $X$,  $1\leqslant\Omega^{\prime}(X)\leqslant\frac{8}{5}$, $\Omega^{\prime}(X)=\frac{9}{10} \gamma_X(\frac{1}{3})$, and $X$ is a Hilbert space if and only if $\Omega^{\prime}(X)=1$.

Later, in order to investigate the distance between isosceles orthogonality and Pythagorean orthogonality, Yang et al. \cite{01} defined a new orthogonal geometric constant, as follows:
$$
\Omega_X(\alpha)=\sup \left\{\frac{\|\alpha x+y\|^2+\|x+\alpha y\|^2}{\|x+y\|^2}: x, y \in X,(x, y) \neq(0,0), x\perp_{I}y\right\} \text {, where } 0 \leqslant \alpha<1 \text {. }
$$

They first gave the upper and lower bounds of the constant $\Omega_X(\alpha)$, namely $1+\alpha^2$ and $2$, and then established the identity of $\Omega_X(\alpha)=\frac{1+\alpha^2}{2} \gamma_X(\frac{1-\alpha}{1+\alpha})$. Finally, building on this foundational identity, they delineated the relationship between  the constant $\Omega_X(\alpha)$ and the intrinsic geometric properties of  Banach spaces. This exploration encompasses uniformly non-square, uniformly smooth,  uniformly convex, and normal structure, etc. 

Motivated by these constants, we define a  novel geometric constant of symmetric form that provides an equivalent definition for the von Neumann-Jordan constant from an orthogonal perspective, as follows:
$$ L_X(t)=\sup \left\{\frac{\|t x+(1-t) y\|^2+\|(1-t) x+t y\|^2}{\|x+y\|^2}: x, y \in X,(x, y) \neq(0,0), x\perp_{I}y\right\}, \text { where } 0 \leqslant t<\frac{1}{2}.
$$

Clearly, the geometric constant $L_X(t)$ has the following equivalent definitions:
$$\begin{aligned}\ L_X(t)&=\sup \left\{\frac{2\left(\|t x+(1-t) y\|^2+\|(1-t) x+t y\|^2\right)}{\|x+y\|^2+\|x-y\|^2}: x, y \in X,(x, y) \neq(0,0), x\perp_{I}y\right\}, \text { where } 0 \leqslant t<\frac{1}{2},\\
	&=\sup \left\{\frac{\|t x+(1-t) y\|^2+\|(1-t) x+t y\|^2}{\|x+y\|^2}: x, y \in X,(x, y) \neq(0,0), x\perp_{I}y\right\}, \text { where } \frac{1}{2} < t\leqslant1.
\end{aligned}$$

This article is organized as follows. 

In the second part, we revisit some basic concepts in Banach spaces and present a variety of widely recognized geometric constants that are associated with the novel geometric constant $L_X(t)$, as well as significant conclusions about them.

In the third part, we discuss in detail the constant $L_X(t)$. We calculate its bounds. We also list three examples where the constant $L_X(t)$ achieves its upper bound in these spaces and show that $L_X(t)$ is a convex and continuous function.

In the fourth part, we first present the identities involving the constant $L_X(t)$ and  the function $\gamma_X(t)$, as well as the von Neumann-Jordan constant. We also calculate a lower bound for the constant $L_X(t)$ in $l_p$ spaces. Additionally, we study some estimates of the constant $L_X(t)$ and other geometric constants.

In the fifth part, we first use the lower bound of the constant $L_X(t)$ to characterize Hilbert spaces. Utilizing the connection between the constant $L_X(t)$ and other geometric constants, we further explore the relationship between the constant $L_X(t)$ and geometric properties of  Banach spaces, such as being uniformly non-square, having a uniformly normal structure, and being uniformly smooth, etc. Finally, we calculate a lower bound for the constant $L_X(t)$  in $l_p - l_q$ spaces, as well as precise values in both $l_2 - l_1$ and $l_{\infty} - l_1$ spaces.

\section{Preliminaries}\label{sec2}
In this article, we consistently  assume that $X$ is a  real Banach space with $\dim X\geqslant 2$, $B_X = \{x\in X: \|x\|\leqslant 1\}$ and $S_X = \{x\in X: \|x\|= 1\}$ will denote the unit ball and the  unit sphere of $X$, respectively.
Now, we will revisit  some concepts of the geometric properties in Banach spaces, along with a discussion on several widely recognized geometric constants.

If there exists $\delta\in(0,1)$, and for any $x, y\in S_X$, it holds that  $\|x + y\|\leqslant 2(1-\delta)$ or $\|x-y\|\leqslant 2(1-\delta)$, then  $X$ is called a uniformly non-square space \cite{12}. Given that for any $\varepsilon> 0$, there exist $x, y\in S_X$ such that $\|x\pm y\|> 2-\varepsilon$, then  $X$ is not uniformly non-square.

A Banach space $ X $ is said to possess a (weak) normal structure \cite{13} if, for every (weakly compact) closed, bounded, and convex subset $ K $ of $ X $ containing more than one point, there exists a point $ x_0 \in K $ such that
$$\sup\{\|x_0-y\| : y \in K\} < d(K) = \sup\{\|x-y\| : x, y \in K\}.$$
Furthermore, a Banach space $ X $ has a uniform normal structure if, there exists a constant $ 0 < c < 1 $ such that for every closed, bounded, and convex subset $ K $ of $ X $ containing more than one point, there exists a point $ x_0 \in K $ such that
$$\sup\{\|x_0-y\| : y \in K\} < cd(K) = c\sup\{\|x-y\| : x, y \in K\}.$$

The concepts of normal and weakly normal structures play a crucial role in fixed point theory. It is clear that in every reflexive Banach space, the property of having a normal structure is tantamount to possessing a weakly normal structure. James  \cite{12} demonstrated  that if a Banach space $X$ is uniformly non-square, then it can be inferred that $X$ is reflexive. Additionally, Kirk \cite{15} demonstrated that if a reflexive Banach space $X$ has a normal structure, then $X$ must possess the fixed point property. It is noteworthy that every uniformly non-square Banach space possesses the fixed point property, a fact that has been confirmed in \cite{14}.

A Banach space $ X $ is called uniformly convex if for any $ 0 < \varepsilon \leqslant 2 $, there exists a $ \delta > 0 $ such that if $ x, y \in S_X $ and $ \|x - y\| \geqslant \varepsilon $, it holds that $\frac{\|x + y\|}{2} \leqslant 1 - \delta.$

A Banach space $ X $ is called strictly convex if for any $x,y\in S_X$ and $x \neq y$, it holds that $\|x + y\| < 2$.

The von Neumann-Jordan constant $C_{\mathrm{NJ}}(X)$ was introduced by Clarkson \cite{00} and is defined as  follows:
$$C_{\mathrm{NJ}}(X)=\sup \left\{\frac{\|x+y\|^2+\|x-y\|^2}{2\left(\|x\|^2+\|y\|^2\right)}:  x, y \in X,(x, y)\neq(0,0)\right\}.$$

Since then, numerous scholars have conducted in-depth research on the von Neumann-Jordan constant and dedicated themselves to its promotion, uncovering many excellent properties. For example,\\
(i) $1\leqslant C_{\mathrm{NJ}}(X)\leqslant2$;\\
(ii) $X$ is a Hilbert space if and only if $C_{\mathrm{NJ}}(X) = 1$;\\
(iii) $X$ is uniformly non-square space if and only if $C_{\mathrm{NJ}}(X) < 2$ \cite{09}.\\
The modified von Neumann-Jordan constant \cite{00} is defined as
$$C_{\mathrm{NJ}}^{\prime}(X)=\sup \left\{\frac{\|x+y\|^2+\|x-y\|^2}{4}:  x, y \in S_X\right\}.$$

In  \cite{02},  P.L. Papini mentioned the following constant:
$$C_{\mathrm{NJ}}^{\prime\prime}(X)=\sup \left\{\frac{\|x+y\|^2+\|x-y\|^2}{2\left(\|x\|^2+\|y\|^2\right)}:  x, y \in X,(x, y)\neq(0,0), x\perp_{I}y\right\}.$$
P.L. Papini showed that the results (ii) and (iii) are also true for the constant $C_{\mathrm{NJ}}^{\prime\prime}(X)$.

The modulus of smoothness  \cite{07} is  defined by the function $\rho(t):$
$$\rho(t)=\sup\left\{\frac{\|x+ty\|+\|x-ty\|}{2}-1:x,y\in S_X\right\}.$$
A Banach space $ X $ is said to be uniformly smooth if $
\lim _{t \rightarrow 0} \frac{\rho(t)}{t}=0.$

In \cite{04}, Yang introduced the function $\gamma_X(t)$: [0,1] $\rightarrow$ [1,4], which is defined as
$$
\gamma_X(t) =\sup \left\{\frac{\|x+t y\|^2+\|x-t y\|^2}{2}: x \in S_X, y \in S_X\right\} .
$$
The function $\gamma_X(t)$ has the following properties:\\
(i) $1 \leqslant 1+t^2 \leqslant \gamma_X(t) \leqslant(1+t)^2 \leqslant 4$.\\
(ii) $X$ is a Hilbert space if and only if $\gamma_X(t) = 1 + t^2$ for any $t\in [0, 1].$\\
(iii) If $2\gamma_X(t) < 1 + (1 + t)^2$ for some $t\in(0, 1]$, then $X$ has uniform normal structure.\\
(iv) If $\lim _{t \rightarrow 0^{+}} \frac{ \gamma_X(t)-1}{t}=0$, then $X$ is uniformly smooth.\\
It is clear that the von Neumann-Jordan constant $C_{\mathrm{NJ}}(X)$ can be defined equivalently as follows:
$$
C_{\mathrm{NJ}}(X)=\sup \left\{\frac{\gamma_X(t)}{1+t^2}: 0 \leqslant t \leqslant 1\right\}.
$$

The modulus of convexity \cite{05} is defined by the function $\delta_X$:
$$\delta_X(\varepsilon)=\inf \left\{1-\frac{\|x+y\|}{2}: x, y \in S_X,\|x-y\|=\varepsilon\right\}, \quad 0\leqslant\varepsilon \leqslant 2.$$
We have compiled a list of properties for this constant, as follows:\\
(i) If $\delta_X(1)> 0$, then $X$ has normal structure \cite{08}.\\
(ii) If $\delta_X(\varepsilon)>0$ for any $\varepsilon \in (0, 2]$, then $X$ is called uniformly convex.\\
(iii) $X$ is strictly convex if and only if $\delta_X(2) = 1$  \cite{10}.\\
(iv) $X$ is uniformly convex if and only if $\sup\{\varepsilon\in [0, 2] :\delta_X(\varepsilon) = 0\} = 0$ \cite{08}.

In \cite{11}, Gao and Lau introduced the James constant, which is defined as
$$J(X) = \sup \{\min\{\|x+y\|, \|x-y\|\} : x, y \in S_X\}.$$
Note that $$\begin{aligned} 
	J(X) &= \sup \{\min\{\|x+y\|, \|x-y\|\} : x, y \in B_X\}\\
	&=\sup \{\|x+y\| : x, y \in S_X, x\perp_{I}y\}.
\end{aligned} $$

\section{Some Bounds and Properties of $L_X(t)$}
First, we compute the bounds of $L_X(t)$.
\begin{proposition}
	Let $X$ ba Banach space, then $2t^2-2t+1\leqslant L_X(t)\leqslant 2t^2-4t+2$.
\end{proposition} 
\begin{proof} 
	Let $x=0,y\neq0$, then $x \perp_I y$ and
	$$\begin{aligned} 
		L_X(t)&\geqslant \frac{\|(1-t) y\|^2+\|t y\|^2}{\|y\|^2}\\
		&=2t^2-2t+1.
	\end{aligned}$$
	On the other hand, 
	$$\begin{aligned} 
		\frac{\|t x+(1-t) y\|^2+\|(1-t) x+t y\|^2}{\|x+y\|^2} &= \frac{\|\frac{1}{2}(x+y)-\frac{1-2t}{2}(x-y)\|^2+\|\frac{1}{2}(x+y)+\frac{1-2t}{2}(x-y)\|^2}{\|x+y\|^2}\\
		&\leqslant \frac{2\left(\frac{1}{2}\|x+y\|+\frac{1-2t}{2}\|x-y\|\right)^2}{\|x+y\|^2}\\
		&=2t^2-4t+2.
	\end{aligned}$$
\end{proof}
In the following  three examples, we demonstrate that the upper bound of $L_X(t)$ is sharp.
\begin{example}\label{3.1}
	Let $X=(\mathbb{R}^{2},\|\cdot\|_{1})$, then $L_X(t)=2t^2-4t+2$.
\end{example}		
\begin{proof}
	Let $x=(1,1)$, $y=(1,-1)$, then $x\perp_{I}y$ and $\|t x+(1-t) y\|_1=\|(1-t) x+ty\|_1=2-2t$. Hence, we have $L_X(t)=2t^2-4t+2$.
\end{proof}		

\begin{example}
	Let $X=(\mathbb{R}^{2},\|\cdot\|_{\infty})$, then $L_X(t)=2t^2-4t+2$.
\end{example}		
\begin{proof}
	Let $x=(1,0)$, $y=(0,-1)$, then $x\perp_{I}y$ and $\|t x+(1-t) y\|_{\infty}=\|(1-t) x+ty\|_{\infty}=1-t$. Hence, we have $L_X(t)=2t^2-4t+2$.
\end{proof}		
\begin{example}		
	Let $C[a,b]$ denote the linear space of all real-valued continuous functions, with the norm defined as:
	$$\|x\|=\sup_{m\in[a,b]} |x(m)|,$$
	then $L_X(t)=2t^2-4t+2$.
\end{example}
\begin{proof}
	Let $x_0=\frac{1}{a-b}(m-b),y_0=\frac{-1}{a-b}(m-b)+1 \in S_{C[a,b]}$, then $x_0\perp_{I} y_0,$ we have
	$$
	\begin{aligned}
		L_X(t)&\geqslant\frac{\|t x_0+(1-t) y_0\|^2+\|(1-t) x_0+t y_0\|^2}{\|x_0+y_0\|^2} \\ &=\sup_{m\in[a,b]}\left|\frac{2t-1}{a-b}(m-b)+1-t\right|^2+\sup_{m\in[a,b]}\left|\frac{1-2t}{a-b}(m-b)+t\right|^2\\
		&=2t^2-4t+2.
	\end{aligned}
	$$
\end{proof}
Finally, we determined that the constant $L_X(t)$ is a convex and continuous function.
\begin{proposition}
	Let $X$ be a Banach space. Then\\
	\rm{(i)} $L_X(t)$ is a convex function.\\
	\rm{(ii)} $L_X(t)$ is continuous on  $[0, \frac{1}{2})$.\\
	\rm{(iii)} If $X$ is a Hilbert space, then $\frac{L_X(t)-1}{1-t}$ is non-increasing on $[0, \frac{1}{2})$.
\end{proposition}	
\begin{proof}
	(i) Let $x\perp_{I}y$, $t_1,t_2 \in [0, \frac{1}{2}), \lambda\in(0, 1)$.\\
	Since $$\|(\lambda t_1+(1-\lambda)t_2)x+(1-(\lambda t_1+(1-\lambda)t_2))y\|=\|\lambda(t_1x+(1-t_1)y)+(1-\lambda)(t_2x+(1-t_2)y)\|,$$
	and $$\|(1-(\lambda t_1+(1-\lambda)t_2))x+(\lambda t_1+(1-\lambda)t_2)y\|=\|\lambda((1-t_1)x+t_1y)+(1-\lambda)((1-t_2)x+t_2y)\|.$$
	Then we have
	$$\begin{aligned}
		&\|(\lambda t_1+(1-\lambda)t_2)x+(1-(\lambda t_1+(1-\lambda)t_2))y\|^2+\|(1-(\lambda t_1+(1-\lambda)t_2))x+(\lambda t_1+(1-\lambda)t_2)y\|^2\\
		&\leqslant\left[\lambda\|t_1x+(1-t_1)y\|+(1-\lambda)\|t_2x+(1-t_2)y\|\right]^2+\left[\lambda\|(1-t_1)x+t_1y\|+(1-\lambda)\|(1-t_2)x+t_2y\|\right]^2\\
		&\leqslant\lambda\left[\|t_1x+(1-t_1)y\|^2+\|(1-t_1)x+t_1y\|^2\right]+(1-\lambda)\left[\|t_2x+(1-t_2)y\|^2+\|(1-t_2)x+t_2y\|^2\right],
	\end{aligned}
	$$
	which implie that 
	$$L_X(\lambda t_1+(1-\lambda)t_2)\leqslant \lambda L_X(t_1)+(1-\lambda)L_X(t_2).$$
	(ii) Obvious.\\
	(iii) Since $X$ is a Hilbert space, we have $L_X(1)=L_X(0)=1$.\\ Let $0 \leqslant t_1<t_2<\frac{1}{2}$, then
	$$
	\begin{aligned}
		\frac{L_X\left(t_2\right)-1}{1-t_2} & =\frac{L_X\left(\frac{t_2-t_1}{1-t_1} \cdot 1+\left(1-\frac{t_2-t_1}{1-t_1}\right) \cdot t_1\right)-1}{1-t_2} \\
		& \leqslant   \frac{\frac{t_2-t_1}{1-t_1} +\left(1-\frac{t_2-t_1}{1-t_1}\right) L_X\left(t_1\right)-1}{1-t_2}\\
		&=\frac{L_X\left(t_1\right)-1}{1-t_1},
	\end{aligned}
	$$
	thus $\frac{L_X(t)-1}{1-t}$ is a non-increasing function.
\end{proof}

\section{$L_X(t)$ in Relation to Other  Geometric Constants}
First, we present the identity for  $L_X(t)$ 
and $\gamma_X(t)$, and on this basis, we establish the identity between $L_X(t)$ and the von Neumann-Jordan constant 	$C_{\mathrm{NJ}}(X)$. This is the core of this article, explaining the von Neumann-Jordan constant from an orthogonal perspective.
\begin{theorem} \label{1123}
	Let $X$ be Banach space, then $L_X(t)=\frac{1}{2} \gamma_X\left(1-2t\right)$.
\end{theorem} 
\begin{proof}[Proof of Theorem~{\upshape\ref{1123}}]
	Let $x, y \in X$ such that $x \perp_I y$, we choose $\alpha=\frac{x+y}{2}, \beta=\frac{x-y}{2}$, then
	$$
	t x+(1-t) y=\alpha-(1-2t)\beta, (1-t) x+t y=\alpha+(1-2t)\beta,
	$$
	thus $\|\alpha\|=\|\beta\|$ and
	$$
	\begin{aligned}
		\frac{\|t x+(1-t) y\|^2+\|(1-t) x+t y\|^2}{\|x+y\|^2} & =\frac{\|\alpha-(1-2t)\beta\|^2+\|\alpha+(1-2t)\beta\|^2}{4\|\alpha\|^2}.
	\end{aligned}
	$$
	Let $x^{\prime}=\frac{\alpha}{\|\alpha\|}, y^{\prime}=\frac{\beta}{\|\beta\|}$, then $x^{\prime}, y^{\prime} \in S_X$ and
	$$
	\begin{aligned}
		\frac{\|\alpha-(1-2t)\beta\|^2+\|\alpha+(1-2t)\beta\|^2}{\|\alpha\|^2} & =\left\|x^{\prime}-(1-2t) y^{\prime}\right\|^2+\left\|x^{\prime}+(1-2t) y^{\prime}\right\|^2 \\
		& \leqslant 2 \cdot\gamma_X\left(1-2t\right),
	\end{aligned}
	$$
	that is, $L_X(t) \leqslant \frac{1}{2} \gamma_X\left(1-2t\right)$.
	
	On the other hand, let $x, y \in S_X$, we choose $\alpha=\frac{x+y}{2}, \beta=\frac{x-y}{2}$, then $\alpha+\beta, \alpha-\beta \in S_X$. Since
	$$
	\begin{aligned}
		\frac{\left\|x-(1-2t)y\right\|^2+\left\|x+(1-2t) y\right\|^2}{2} & =\frac{\left\|\alpha+\beta-(1-2t)(\alpha-\beta)\right\|^2+\left\|\alpha+\beta+(1-2t)(\alpha-\beta)\right\|^2}{2\|\alpha+\beta\|^2} \\
		& =2 \cdot \frac{\|t\alpha+(1-t)\beta\|^2+\|(1-t) \alpha+t\beta\|^2}{\|\alpha+\beta\|^2} \\
		& \leqslant 2L_X(t),
	\end{aligned}$$
	that is, $L_X(t) \geqslant \frac{1}{2} \gamma_X\left(1-2t\right).$
\end{proof}
\begin{remark}
	Since $\gamma_X(t)$ is a non-decreasing function, then it's easy to know $L_X(t)$ is a non-increasing on  $[0, \frac{1}{2})$.
\end{remark}
\begin{theorem} \label{thm4}
	Let $X$ be Banach space, then $C_{\mathrm{NJ}}(X)=\sup \left\{\frac{2L_X\left(\frac{1-\eta}{2}\right)}{1+\eta^2}: 0\leqslant  \eta\leqslant 1\right\}$.
\end{theorem} 
\begin{proof}[Proof of Theorem~{\upshape\ref{thm4}}]
	First, we need to expand upon the definition with $L_X(\frac{1}{2})= \frac{1}{2}$.	Combined with Theorem \ref{1123}, we have  $$L_X(t)=\frac{1}{2}\cdot \gamma_X\left(1-2t\right),$$
	let $1-2t=\eta$,  then $\eta\in[0,1]$ and
	$$ \gamma_X\left(\eta\right)=2\cdot L_X\left(\frac{1-\eta}{2}\right).$$
	Hence, we can obtain that
\begin{equation}
		C_{\mathrm{NJ}}(X)=\sup \left\{\frac{2L_X\left(\frac{1-\eta}{2}\right)}{1+\eta^2}: 0\leqslant  \eta\leqslant 1\right\}.\label{eq01}
\end{equation}
\end{proof}
In the following two examples, we estimate the value of the geometric constant $L_X(t)$ in the $l_p$ space, and use the identity of (\ref{eq01}) to calculate the value of the von Neumann-Jordan constant in $X=(\mathbb{R}^{2},\|\cdot\|_{1})$, which is in agreement with the known result, indicating that the identity is meaningful.
\begin{example} Let $l^p$ ($ 1 < p < \infty $) be the linear space of all sequences in $\mathbb{R}$ for which $\sum_{i=1}^\infty |x_i|^p < \infty$. The norm is defined as:
	$$
	\|x\|_p=\left(\sum_{i=1}^{\infty}\left|x_i\right|^p\right)^{\frac{1}{p}},
	$$
	for any sequence $x = (x_i) \in l^p$.
	Then 	$$
	\ L_{l_p}(t) \geqslant 2^{1-\frac{2}{p}}\left((1-t)^p+t^p\right)^{\frac{2}{p}} .
	$$
	Particularly, if $2 \leqslant p<\infty$, then 	$$
	L_{l_p}(t) = 2^{1-\frac{2}{p}}\left((1-t)^p+t^p\right)^{\frac{2}{p}} .
	$$
\end{example}
\begin{proof}
	Let $x=\left(\frac{1}{2^{\frac{1}{p}}}, \frac{1}{2^{\frac{1}{p}}}, 0, \cdots, 0\right), y=\left(\frac{1}{2^{\frac{1}{p}}},-\frac{1}{2^{\frac{1}{p}}}, 0, \cdots, 0\right)$, then $x,y \in S_X$ and
	$$
	\gamma_{l_p}(t) \geqslant \frac{\left\|x+t y\right\|^2+\left\|x-t y\right\|^2}{2}=\left(\frac{(1+t)^p+(1-t)^p}{2}\right)^{\frac{2}{p}},
	$$
	hence
	$$
	L_{l_p}(t) \geqslant \frac{1}{2} \cdot\left(\frac{\left(2-2t\right)^p+\left(2t\right)^p}{2}\right)^{\frac{2}{p}}=2^{1-\frac{2}{p}}\left((1-t)^p+t^p\right)^{\frac{2}{p}}.$$
	Particularly, if $2\leqslant p<\infty$, then  $\gamma_{l_p}(t)=\left(\frac{(1+t)^p+(1-t)^p}{2}\right)^{\frac{2}{p}}$ \cite{04}, so we can infer that
	$$
	L_{l_p}(t) = 2^{1-\frac{2}{p}}\left((1-t)^p+t^p\right)^{\frac{2}{p}} .
	$$
\end{proof}
\begin{example}
	Let $X=(\mathbb{R}^{2},\|\cdot\|_{1})$, then $C_{\mathrm{NJ}}(X)=2$.
\end{example}
\begin{proof}
	In Example \ref{3.1}, we have $$L_X(t) = 2t^2-4t+2,$$
	which implies that $$C_{\mathrm{NJ}}(X)=\sup_{\eta\in[0,1]}\left\{\frac{4\left(\frac{1-\eta}{2}-1\right)^2}{1+\eta^2}\right\}=2.$$
\end{proof}
Finally, we estimate the relationship between  $L_X(t)$ and $C_{\mathrm{NJ}}^{\prime}(X), J(X),$ and $\delta_X(\varepsilon)$ in turn.
\begin{proposition}
	Let $X$ be Banach space, then $L_X(t)\geqslant(4t^2-4t+1)C_{\mathrm{NJ}}^{\prime}(X).$
\end{proposition}
\begin{proof}
	For any $x,y \in S_X$, it holds that $x+y\perp_{I}x-y$, then
	$$
	L_{X}(t)\geqslant\frac{\|t (x+y)+(1-t) (x-y)\|^2+\|(1-t) (x+y)+t  (x-y)\|^2}{\|(x+y)+(x-y)\|^2}.
	$$
	Since $$\|t (x+y)+(1-t) (x-y)\|\geqslant \biggl|\|t(x+y)\|-\|(1-t)(x-y)\|\biggr|$$
	and $$\|(1-t) (x+y)+t  (x-y)\|\geqslant \biggl|\|(1-t)(x+y)\|-\|t(x-y)\|\biggr|,$$
	then $$
	\begin{aligned}
		L_{X}(t)&\geqslant\frac{(t^2+(1-t)^2)(\|x+y\|^2+\|x-y\|^2)-4t(1-t)\|x+y\|\|x-y\|}{4}\\
		&\geqslant\frac{(1-2t+t^2)(\|x+y\|^2+\|x-y\|^2)-2t(1-t)(\|x+y\|^2+\|x-y\|^2)}{4}\\
		&=(4t^2-4t+1)\cdot\frac{\|x+y\|^2+\|x-y\|^2}{4},
	\end{aligned}$$
	which implies that
	$	L_{X}(t)\geqslant(4t^2-4t+1)C_{\mathrm{NJ}}^{\prime}(X).$
\end{proof}
\begin{proposition}
	Let $X$ be Banach space, then $$\frac{1}{2}J^2(X)-2tJ(X)+2t^2\leqslant L_X(t)\leqslant\frac{2t^2-2t+1}{4}J^2(X)+(2t-2t^2)J(X)+2t^2-2t+1.$$
\end{proposition}
\begin{proof}
	For any $x,y\in S_X$, since $$\|x+y\|=\|x+(1-2t)y+2ty\|\leqslant\|x+(1-2t)y\|+2t$$ and  $$\|x-y\|=\|x-(1-2t)y-2ty\|\leqslant\|x-(1-2t)y\|+2t,$$
	so we obtain  
	$$
	\begin{aligned}
		& \min \{\|x+y\|,\|x-y\|\}^2 \\
		& \leqslant \min \left\{\|x+(1-2t)y\|+2t,\|x-(1-2t)y\|+2t\right\}^2 \\
		& \leqslant \frac{\left(\|x+(1-2t)y\|+2t\right)^2+\left(\|x-(1-2t)y\|+2t\right)^2}{2} \\
		& =\frac{\|x+(1-2t)y\|^2+\|x-(1-2t)y\|^2}{2}+\frac{4t(\|x+(1-2t)y\|+\|x-(1-2t)y\|)}{2}+4t^2 \\
		& \leqslant \gamma_X\left(1-2t\right)+4t \sqrt{\gamma_X\left(1-2t\right)}+4t^2\\
		&=\left(\sqrt{\gamma_X\left(1-2t\right)}+2t\right)^2,
	\end{aligned}
	$$
	which implies that $J(X)\leqslant\sqrt{\gamma_X\left(1-2t\right)}+2t=\sqrt{2L_X(t)}+2t$, namely
	$$\frac{1}{2}J^2(X)-2tJ(X)+2t^2\leqslant L_X(t).$$
	On the other hand, since $$\|x+\alpha y\|=\left\|\frac{1+\alpha}{2}(x+y)+\frac{1-\alpha}{2}(x-y)\right\| \leqslant \frac{1+\alpha}{2}\|x+y\|+\frac{1-\alpha}{2}\|x-y\|$$ and $$\|x-\alpha y\|=\left\|\frac{1-\alpha}{2}(x+y)+\frac{1+\alpha}{2}(x-y)\right\| \leqslant \frac{1-\alpha}{2}\|x+y\|+\frac{1+\alpha}{2}\|x-y\|,$$ we have
	$$
	\begin{aligned}
		\|x+\alpha y\|^2+\|x-\alpha y\|^2 &\leqslant \frac{1+\alpha^2}{2}\left(\|x+y\|^2+\|x-y\|^2\right)+\left(1-\alpha^2\right)\|x+y\|\|x-y\| \\
		& \leqslant \frac{1+\alpha^2}{2}\left(J^2(X)+4\right)+2\left(1-\alpha^2\right) J(X),
	\end{aligned}
	$$
	which implies that $\gamma_X\left(1-2t\right) \leqslant \frac{4t^2-4t+2}{4} J^2(X)+(4t-4t^2)J(X)+4t^2-4t+2$. Hence $$L_X(t)\leqslant\frac{2t^2-2t+1}{4}J^2(X)+(2t-2t^2)J(X)+2t^2-2t+1.$$
\end{proof}
\begin{theorem}\label{11}
	Let $X$ be Banach space, then $$ \begin{aligned}
		\frac{1}{2} \left(1-2t\right)^2 \left(\frac{\varepsilon}{2}-\delta_X(\varepsilon)+1\right)^2&\leqslant L_X(t)\\
		&\leqslant\left(2t^2-2t+1\right)\left(1-\delta_X(\varepsilon)\right)^2+2t(1-t) \varepsilon\left(1-\delta_X(\varepsilon)\right)+\frac{2t^2-2t+1}{4} \varepsilon^2.\end{aligned}$$
\end{theorem}

\begin{proof} [Proof of Theorem~{\upshape\ref{11}}]
	For any $x,y\in S_X$, it holds that
	$$
	\begin{aligned}
		\gamma_X(t) & \geqslant\frac{\|x+t y\|^2+\|x-t y\|^2}{2} \\
		& \geqslant \frac{1+t^2}{4}\left(\|x+y\|^2+\|x-y\|^2\right)-\frac{\left(1-t^2\right)\|x+y\|\|x-y\|}{2} \\
		& \geqslant t^2\left(\frac{\|x+y\|+\|x-y\|}{2}\right)^2.
	\end{aligned}
	$$
	Combined with the definition of the constant $\rho_X(1)$, we have $\gamma_X(t) \geqslant t^2\left(\rho_X(1)+1\right)^2$, hence
	$$
	\ L_X(t)=\frac{1}{2} \gamma_X\left(1-2t\right) \geqslant \frac{1}{2}(1-2t)^2\left(\rho_X(1)+1\right)^2 .
	$$
	In  \cite{07}, it is noted that $\rho_X(1)=\sup \left\{\frac{\varepsilon}{2}-\delta_X(\varepsilon): 0 \leqslant \varepsilon \leqslant 2\right\}$, so we obtain
	$$
	L_X(t) \geqslant \frac{1}{2} \left(1-2t\right)^2 \left(\frac{\varepsilon}{2}-\delta_X(\varepsilon)+1\right)^2 .
	$$
	On the other hand, for any $x, y \in S_X$,	it holds that $\|x+y\| \leqslant 2-2 \delta_X(\varepsilon)$ , then
	$$
	\begin{aligned}
		\frac{\|x+t y\|^2+\|x-t y\|^2}{2}  &\leqslant \frac{1+t^2}{4}\left(\|x+y\|^2+\|x-y\|^2\right)+\frac{1-t^2}{2}\|x+y\|\|x-y\| \\
		& \leqslant \frac{1+t^2}{4}\left(4\left(1-\delta_X(\varepsilon)\right)^2+\varepsilon^2\right)+\left(1-t^2\right)\left(1-\delta_X(\varepsilon)\right)\varepsilon,
	\end{aligned}
	$$
	which implies that $\gamma_X(t) \leqslant\left(1+t^2\right)\left(1-\delta_X(\varepsilon)\right)^2+\left(1-t^2\right)\varepsilon\left(1-\delta_X(\varepsilon)\right) +\frac{1+t^2}{4} \varepsilon^2$. 
	Therefore, $$L_X(t)=\frac{1}{2} \gamma_X\left(1-2t\right) \leqslant \left(2t^2-2t+1\right)\left(1-\delta_X(\varepsilon)\right)^2+2t(1-t) \varepsilon\left(1-\delta_X(\varepsilon)\right)+\frac{2t^2-2t+1}{4} \varepsilon^2.$$
\end{proof}

\section{Some Geometric Properties Associated with $L_X(t)$}
First, we use the lower bound of the geometric constant $L_X(t)$ to characterize Hilbert spaces, showing that the lower bound of $L_X(t)$ is also sharp, and provid three equivalent forms. Here, we introduce two lemmas that will be used in the proof.
\begin{lemma}
	\rm\cite{17}	A normed space $X$ is an inner product space if and only if $x \perp_I y \Longleftrightarrow x \perp_P y$ for all $x, y\in X$.
\end{lemma}
\begin{lemma}\label{Lemma2}
	\rm\cite{16}	A normed space $X$ is an inner	product space if and only if for any $x, y \in S_X$, there exist $a, b \neq 0$ such that
	$$
	\|a x+b y\|^2+\|a x-b y\|^2 \sim 2\left(a^2+b^2\right)
	$$
	where $\sim$ stands for $=, \leqslant$ or $\geqslant$.
\end{lemma}
\begin{theorem}\label{thm8} Let $X$ be Banach space, then the following conditions are equivalent:\\
	\rm{(i)} $X$ is Hilbert space.\\
	\rm{(ii)}  $L_X(t)= 2t^2-2t+1$ for any $t\in [0,\frac{1}{2})$.\\
	\rm{(iii)} $L_X(t)= 2t_0^2-2t_0+1$ for some $t_0\in [0,\frac{1}{2})$.
\end{theorem}
\begin{proof} [Proof of Theorem~{\upshape\ref{thm8}}]
	 Assuming (i) holds, then for any two non-zero vectors $x,y\in X$ such that $x\perp_{I}y$,  it follows that $x\perp_{P}y$, namely $\|x+y\|^2=\|x\|^2+\|y\|^2$. Therefore
	$$
	\begin{aligned}
		\|t x+(1-t) y\|^2+\|(1-t) x+t y\|^2&=\|tx\|^2+\|(1-t)y\|^2+\|(1-t)x\|^2+\|ty\|^2+4t(1-t)\langle x,y \rangle \\
		&=(2t^2-2t+1)(\|x\|^2+\|y\|^2),
	\end{aligned}
	$$
	which implie that 
	$$
	\begin{aligned}
		\frac{\|t x+(1-t) y\|^2+\|(1-t) x+t y\|^2}{\|x+y\|^2}&=\frac{(2t^2-2t+1)(\|x\|^2+\|y\|^2)}{\|x\|^2+\|y\|^2}\\
		&=2t^2-2t+1.
	\end{aligned}
	$$
	Assuming (ii) holds, then (iii) clearly holds.\\
	Assuming (iii) holds, then  for any $x,y \in S_X$, it holds that $x+y\perp_{I} x-y$. Therefore
	$$\frac{\|t_0 (x+y)+(1-t_0) (x-y)\|^2+\|(1-t_0) (x+y)+t_0  (x-y)\|^2}{\|(x+y)+(x-y)\|^2}\leqslant 2t_0^2-2t_0+1,$$
	namely $\|x-(1-2t_0)y\|^2+\|x+(1-2t_0)y\|^2\leqslant 4(2t_0^2-2t_0+1)$ is established. Let
	$a=1, b=(1-2t_0)$,  then $a, b\neq0$ and $\|ax+by\|^2+\|ax-by\|^2\leqslant 2(a^2+b^2)$. Combining Lemma \ref{Lemma2}, then (i) holds.
\end{proof}
Next, we demonstrate that if $X$ is not a uniformly non-square, $L_X(t)$ attains its upper bound. Furthermore, when $L_X(t_0)$ is less than a certain value for some $t_0$, $X$ exhibits a (uniformly) normal structure.
\begin{theorem}\label{123} Let $X$ be Banach space, then\\	
	\rm{(i)} If $X$ is not uniformly non-square, then $L_X(t)=2t^2-4t+2$ for any $t \in[0,\frac{1}{2})$.\\
	\rm{(ii)} If $L_X(t_0)<\frac{4t_0^2-8t_0+5}{4}$ for some $t_0 \in[0,\frac{1}{2})$, then $X$ has uniform normal structure.\\
	\rm{(iii)} If $L_X\left(t_0\right)<\frac{9\left(1-2t_0\right)^2}{8}$ for some $t_0 \in[0,\frac{1}{2})$, then $X$ has normal structure.
\end{theorem}
\begin{proof} [Proof of Theorem~{\upshape\ref{123}}] (i) Since $X$ is not uniformly non-square, then $\gamma_X\left(1-2t\right)=\left(1+(1-2t)\right)^2=4(1-t)^2$, namely, $L_X(t)=\frac{1}{2} \gamma_X\left(1-2t\right)=2t^2-4t+2$.\\
	(ii) Since $L_X(t_0)<\frac{4t_0^2-8t_0+5}{4}$, then
	$
	2 \gamma_X\left(1-2t\right)=4L_X(t)<4t_0^2-8t_0+5=1+\left(1+(1-2t_0)\right)^2,
	$
	which implies that $X$ has uniform normal structure.\\
	(iii) Since $L_X\left(t_0\right)<\frac{9\left(1-2t_0\right)^2}{8}$, then $\frac{1}{2}(1-2t_0)^2\left(\frac{\varepsilon}{2}-\delta_X(\varepsilon)+1\right)^2<\frac{9\left(1-2t_0\right)^2}{8}$, namely, $\delta_X(\varepsilon)>\frac{\varepsilon-1}{2}$, then $\delta_X(1)>0$. Therefore, $X$ has normal structure.
\end{proof}
Now, we  make use of the following lemma, which is indispensable in the proof of Theorem \ref{Theorem5.3}.
\begin{lemma}\label{Lemma5.3}
	\rm\cite{06} Let $X$ be  Banach space and $x, y \in X$. If $x \perp_{I} y$, then the following inequalities hold.\\
	(i) $\|x+\alpha y\| \leqslant|\alpha|\|x \pm y\|$ and $\|x \pm y\| \leqslant\|x+\alpha y\|$, when $|\alpha| \geqslant 1$.\\
	(ii) $\|x+\alpha y\| \leqslant\|x \pm y\|$ and $|\alpha|\|x \pm y\| \leqslant\|x+\alpha y\|$, when $|\alpha| \leqslant 1$.
\end{lemma}
In Theorem \ref{123}(i), this does not appear to be a necessary and sufficient condition. In contrast, we derive the following conclusion:
\begin{theorem}\label{Theorem5.3} Let $X$ be a finite dimensional Banach space. If $L_X\left(t_0\right)=2t_0^2-4t_0+2$ for some $t_0 \in[0,\frac{1}{2})$, then $X$ is not uniformly non-square.
\end{theorem}
\begin{proof}[Proof of Theorem~{\upshape\ref{Theorem5.3}}] Since $L_X\left(t_0\right)=2t_0^2-4t_0+2$, then there exist $x_n \in S_X, y_n \in B_X$ such that $x_n \perp_I y_n$ and
	$$
	\lim _{n \rightarrow \infty} \frac{\left\|t_0x_n+(1-t_0) y_n\right\|^2+\left\|(1-t_0) x_n+t_0y_n\right\|^2}{\left\|x_n+y_n\right\|^2}=2t_0^2-4t_0+2 .
	$$
	Since $X$ is finite dimensional, then there exist $x_0, y_0 \in B_X$ such that $x_0 \perp_I y_0$ and
	$$
	\lim _{i \rightarrow \infty}\left\|x_{n_i}\right\|=\left\|x_0\right\|, \lim _{i \rightarrow \infty}\left\|y_{n_i}\right\|=\left\|y_0\right\|.
	$$
	Combining Lemma \ref{Lemma5.3},  we obtain $$\left\|t_0x_n+(1-t_0) y_n\right\| \leqslant(1-t_0)\left\|x_n+y_n\right\|\text{ and } \left\|(1-t_0) x_n+t_0y_n\right\| \leqslant(1-t_0)\left\|x_n+y_n\right\|,$$ thus,
	$$
	\frac{(1-t_0)^2\left\|x_n+y_n\right\|^2+(1-t_0)^2\left\|x_n+y_n\right\|^2}{\left\|x_n+y_n\right\|^2} \leqslant 2t_0^2-4t_0+2 ,
	$$
	then $\left\|t_0x_0+(1-t_0)y_0\right\|=(1-t_0)\left\|x_0+y_0\right\|$ and $\left\|(1-t_0) x_0+t_0y_0\right\|=(1-t_0)\left\|x_0+y_0\right\|$.\\
	On the other hand, $\left\|t_0x_0+(1-t_0)y_0\right\| \leqslant\left(1-2t_0\right)\left\|y_0\right\|+t_0\left\|x_0+y_0\right\|$, then $\left\|x_0+y_0\right\| \leqslant\left\|y_0\right\|$. Similarly,  
	$\left\|x_0+y_0\right\| \leqslant\left\|x_0\right\|$ also holds. Therefore,
	$$
	\max \left\{\left\|x_0+y_0\right\|,\left\|x_0-y_0\right\|\right\}=\left\|x_0+y_0\right\| \leqslant \min \left\{\left\|x_0\right\|,\left\|y_0\right\|\right\} \leqslant 1<1+\delta
	$$
	for any $\delta \in(0,1)$, which shows that $X$ is not uniformly non-square.
\end{proof}
Additionally, we apply the inequation between $L_X(t)$ and the modulus of convexity as described in Theorem \ref{11}, along with the properties of the modulus of convexity, to elucidate the relationship between $L_X(t)$ and both uniformly convex spaces and strictly convex spaces.
\begin{proposition}
	Let $X$ be Banach space and $L_X(0) > 1$, then $X$ is not uniformly
	convex. Particularly, if $L_X(t) >2t^2-2t+1$ for any $t\in [0, \frac{1}{2})$, then $X$ is not strictly convex.
\end{proposition}
\begin{proof}
	If $L_X(0) > 1$, then there exists $\varepsilon_0 \in (0, 2]$ such that $L_X(0) \geqslant 1 + \frac{\varepsilon_0^2}{4}$.\\
	Since $L_X(t)\leqslant\left(2t^2-2t+1\right)\left(1-\delta_X(\varepsilon)\right)^2+2t(1-t) \varepsilon\left(1-\delta_X(\varepsilon)\right)+\frac{2t^2-2t+1}{4} \varepsilon^2,$ then
	$$1 + \frac{\varepsilon_0^2}{4}\leqslant\left(1-\delta_X(\varepsilon_0)\right)^2+\frac{\varepsilon_0^2}{4},$$
	we can obtain $\delta_X(\varepsilon_0)= 0$. Therefore, $\sup\left\{\varepsilon\in[0, 2] : \delta_X(\varepsilon) = 0\right\} \geqslant \varepsilon_0 > 0$, $X$ is not
	uniformly convex.
	
	Particularly, if $L_X(t) >2t^2-2t+1$ for any $t\in [0, \frac{1}{2})$, we suppose that $X$ is strictly convex, then $\delta_X(2)= 1$. We have
	$$2t^2-2t+1<\left(2t^2-2t+1\right)\left(1-1\right)^2+4t(1-t) \left(1-1\right)+\frac{2t^2-2t+1}{4}\cdot2^2=2t^2-2t+1,$$
	this is a contradiction,   $X$ is not strictly convex.		
\end{proof}
Now, we discuss the relationship between $L_X(t)$ and uniformly smooth spaces.
\begin{theorem}\label{thm12} Let $X$ be Banach space, then $X$ is uniformly smooth if 
	$$
	\lim _{t \rightarrow \frac{1}{2}^{-}} \frac{2L_X(t)-1}{1-2t}=0 .
	$$
\end{theorem}
\begin{proof}[Proof of Theorem~{\upshape\ref{thm12}}] Let $\alpha=1-2t$, then $t=\frac{1-\alpha}{2}$ and $t \rightarrow \frac{1}{2}^{-} \Longleftrightarrow \alpha\rightarrow 0^{+}$. We have
	$$
	\begin{aligned}
		\frac{2L_X(t)-1}{1-2t}& =\frac{ \gamma_X\left(1-2t\right)-1}{1-2t} \\
		& =\frac{ \gamma_X(\alpha)-1}{\alpha}.
	\end{aligned}
	$$
	Therefore, $\lim _{\alpha \rightarrow 0^{+}} \frac{ \gamma_X(\alpha)-1}{\alpha}=0$, then $X$ is uniformly smooth.
\end{proof}
Finally, we  compute a lower bound for the geometric constant $L_X(t)$ in  $l_p-l_q (1 \leqslant q \leqslant p<\infty)$ spaces and compute its exact values in special $l_2-l_1$ and $l_{\infty}-l_1$ spaces. We also note that $l_2-l_1$ and $l_{\infty}-l_1$  are uniformly non-square.
\begin{example} Let $l_p-l_q(1 \leqslant q \leqslant p<\infty)$ be $\mathbb{R}^{2}$ with the norm defined as:
	$$
	\left\|\left(x_1, x_2\right)\right\|=\left\{\begin{array}{l}
		\left\|\left(x_1, x_2\right)\right\|_p, \quad x_1 x_2 \geqslant 0, \\
		\left\|\left(x_1, x_2\right)\right\|_q, \quad x_1 x_2<0,
	\end{array} \right.
	$$
	then  $$L_{l_p-l_q}(t)\geqslant 2^{-1-\frac{2}{p}}\left[\left(1+2^{\frac{1}{p}-\frac{1}{q}}-2^{\frac{1}{p}-\frac{1}{q}+1}\cdot t \right)^p+\left(1-2^{\frac{1}{p}-\frac{1}{q}}+2^{\frac{1}{p}-\frac{1}{q}+1}\cdot t\right)^p\right]^{\frac{2}{p}}.$$
	Particularly, $
	L_{l_2-l_1}(t)=2t^2-3t+\frac{3}{2}$ and $ L_{l_{\infty}-l_1}(t)=\frac{4t^2-8t+5}{4}.$
\end{example}
\begin{proof}
	Let $x=\left(\frac{1}{2^{\frac{1}{p}}}, \frac{1}{2^{\frac{1}{p}}}\right), y=\left(\frac{1}{2^{\frac{1}{q}}},-\frac{1}{2^{\frac{1}{q}}}\right)$, then $x\perp_{I}y$ and 
	$$
	\gamma_{l_p-l_q}(t) \geqslant \frac{\left\|x+t y\right\|^2+\left\|x-t y\right\|^2}{2}=2^{-\frac{2}{p}}\left[\left(1+t \cdot 2^{\frac{1}{p}-\frac{1}{q}}\right)^p+\left(1-t \cdot 2^{\frac{1}{p}-\frac{1}{q}}\right)^p\right]^{\frac{2}{p}},
	$$
	hence
	$$
	\begin{aligned}
		L_{l_p-l_q}(t)&=\frac{1}{2} \gamma_{l_p-l_q}\left(1-2t\right) \\
		& \geqslant 2^{-1-\frac{2}{p}}\left[\left(1+2^{\frac{1}{p}-\frac{1}{q}}-2^{\frac{1}{p}-\frac{1}{q}+1}\cdot t \right)^p+\left(1-2^{\frac{1}{p}-\frac{1}{q}}+2^{\frac{1}{p}-\frac{1}{q}+1}\cdot t\right)^p\right]^{\frac{2}{p}}.
	\end{aligned}
	$$
	
	Particularly, since $\gamma_{l_2-l_1}(t)=1+t+t^2$ and $\gamma_{l_{\infty}-l_1}(t)=\frac{1}{2}\left(1+(1+t)^2\right)$ \cite{04}, then
	$$
	L_{l_2-l_1}(t)=\frac{1}{2} \gamma_{l_2-l_1}\left(1-2t\right)=2t^2-3t+\frac{3}{2}
	$$
	and
	$$
	L_{l_{\infty}-l_1}(t)=\frac{1}{2} \gamma_{l_{\infty}-l_1}\left(1-2t\right)=\frac{4t^2-8t+5}{4}.
	$$
\end{proof}
By simple calculation, we  obtain $L_{l_2-l_1}(t), L_{l_{\infty}-l_1}(t)<2t^2-4t+2$ for any $t \in[0,\frac{1}{2})$, combined with Theorem \ref{123}  (i), this shows  that $l_2-l_1$, $l_{\infty}-l_1$ both are uniformly non-square.
\begin{remark}\label{remark}
	Using the identity of \rm{(\ref{eq01})}, we  calculate that 
	$$C_{\mathrm{NJ}}(l_2-l_1)=\frac{3}{2},\quad C_{\mathrm{NJ}}(l_{\infty}-l_1)=\frac{3+\sqrt{5}}{4}.$$
\end{remark}
In Remark \ref{remark}, we calculate the value of $C_{\mathrm{NJ}}(X)$ in these two specific spaces, which is consistent with the known results. This further demonstrates that the definition of the geometric constant $L_X(t)$ is meaningful and the identity is valid.

%\section{Results}\label{sec3.11}

%Sample body text. Sample body text. Sample body text. Sample body text. Sample body text. Sample body text. Sample body text. Sample body text.

%\section{This is an example for first level head---section head}\label{sec3}

%\subsection{This is an example for second level head---subsection head}\label{subsec2}

%\subsubsection{This is an example for third level head---subsubsection head}\label{subsubsec2}

%Sample body text. Sample body text. Sample body text. Sample body text. Sample body text. Sample body text. Sample body text. Sample body text. 

%\backmatter

%\bmhead{Acknowledgements}

%Acknowledgements are not compulsory. Where included they should be brief. Grant or contribution numbers may be acknowledged.

%Please refer to Journal-level guidance for any specific requirements.

\section*{Declarations}

%Some journals require declarations to be submitted in a standardised format. Please check the Instructions for Authors of the journal to which you are submitting to see if you need to complete this section. If yes, your manuscript must contain the following sections under the heading `Declarations':

\begin{itemize}
\item Funding This research work was funded by Anhui Province Higher Education Science Research Project (Natural Science), 2023AH050487.
\item Conflict of interest The authors declare that there is no conflict of interest in publishing the article.
%\item Ethics approval and consent to participate
%\item Consent for publication
%\item Data availability 
\end{itemize}

\end{document}